\documentclass[12pt,reqno]{amsart}
\usepackage{cmap,mathtools}
\usepackage[T1]{fontenc}
\usepackage[utf8]{inputenc}
\usepackage[english]{babel}
\usepackage{amsfonts,amssymb,amsmath}
\usepackage{footnote}\usepackage{bigints}
\usepackage{array}
\usepackage{adjustbox}

\usepackage{graphicx}
\usepackage{caption}

 \usepackage{enumitem}

\allowdisplaybreaks

\usepackage{graphicx}
\usepackage{epstopdf}
\usepackage{a4wide}
\usepackage{subfig}
\usepackage{xcolor}
\usepackage[colorlinks=true,linktocpage,pdfpagelabels,
bookmarksnumbered,bookmarksopen]{hyperref}
\definecolor{ForestGreen}{rgb}{0.1,0.6,0.05}
\definecolor{EgyptBlue}{rgb}{0.063,0.1,0.6}
\hypersetup{
	colorlinks=true,
	linkcolor=EgyptBlue,         
	citecolor=ForestGreen,
	urlcolor=olive
}

\usepackage[hyperpageref]{backref}

\newtheorem{theorem}{Theorem}[section]
\newtheorem{proposition}{Proposition}[section]
\newtheorem{definition}{Definition}[section]
\newtheorem{lemma}{Lemma}[section]
\newtheorem{remark}{Remark}[section]
\newtheorem{cor}{Corollary}[section]

\numberwithin{equation}{section}
\numberwithin{theorem}{section}
\numberwithin{equation}{section}
\numberwithin{theorem}{section}

\usepackage{ulem}
\usepackage{xcolor}
\usepackage[colorlinks=true,linktocpage,pdfpagelabels,
bookmarksnumbered,bookmarksopen]{hyperref}
\definecolor{ForestGreen}{rgb}{0.1,0.6,0.05}
\definecolor{EgyptBlue}{rgb}{0.063,0.1,0.6}
\hypersetup{
	colorlinks=true,
	linkcolor=EgyptBlue,         
	citecolor=ForestGreen,
	urlcolor=olive
}

\usepackage[hyperpageref]{backref}
\usepackage[foot]{amsaddr}

\DeclareUnicodeCharacter{2212}{-}

\title [Steklov-Robin eigenvalues on Symmetric domains]{Sharp Bounds for higher  mixed Steklov--Robin eigenvalues on domains with holes}

\author{Sagar Basak }
\address{Department of Mathematical Sciences, Indian Institute of Technology (BHU), Varanasi, India}
\email{sagarbasak.rs.mat22@itbhu.ac.in}

\DeclareUnicodeCharacter{2212}{-}
\begin{document}
\begin{abstract}
This article is concerned with mixed Steklov--Robin eigenvalues on bounded domains in $\mathbb{R}^{n}, n \geq 2$, with Lipschitz boundary. Specifically, we consider domains with symmetry of order $4$ containing a spherical hole. We obtain isoperimetric inequalities for the $k$-th Steklov-Robin eigenvalues for each $k \in \{2, 3, \dots, n+1\}$.  We provide examples to emphasize the fact that the symmetry assumptions, on the family of domains considered, are crucial.
\end{abstract}

\maketitle

\section{introduction} 
 Let  $\Omega_{out}$ 
denote a bounded domain in the $\mathbb{R}^n$ 
with Lipschitz  boundary denoted by $\partial\Omega_{out}$. Let $B_{r} \subset  M$ be a ball of radius ${r}$ such that $\overline{B_{r}} \subset \Omega_{out}$. Consider the following mixed Steklov-Robin eigenvalue problem on $\Omega := \Omega_{out} \setminus \overline{B_{r}}.$
\begin{align} \label{eqn: SD problem}
\begin{cases}
        \Delta u=0 & \text{in} \,\, \Omega,\\
        \frac{\partial u}{\partial \nu}+\beta(x)u=0 & \text{on}\,\, \partial B_{r},\\
        \frac{\partial u}{\partial \nu }=\sigma u & \text{on} \,\, \partial \Omega_{out}. 
    \end{cases}    
\end{align}
Here, $\beta(x) \in L^{\infty}(\partial B_r)$ is a positive function satisfying $\|\beta\|_{L^1(\partial B_r)} = m > 0$, and $\nu$ denotes the outward unit normal vector to $\partial\Omega$.
The mixed Steklov--Robin spectrum coincides with the spectrum of the corresponding Dirichlet to Neumann type operator. More precisely, this operator maps each function $f\in L^{2}(\partial\Omega_{out})$ to the normal derivative $\frac{\partial \tilde{f}}{\partial \nu}\in L^{2}(\partial\Omega_{out})$ of its harmonic extension $\tilde{f}$ into $\Omega$, where $\tilde{f}$ satisfies the Robin boundary condition on $\partial B_r$. Since this Dirichlet to Neumann type operator is self-adjoint and positive, it admits a discrete sequence of real positive eigenvalues,
\begin{align*}
    0<\sigma_1(\Omega)\leq\sigma_2(\Omega)\leq \sigma_3(\Omega) \leq \cdots 
    \longrightarrow \infty ,
\end{align*}
each repeated according to its multiplicity.

For each $k\geq 1,$  mixed Steklov-Robin eigenvalues $\sigma_k(\Omega)$ are determined by the following variational characterization: 
\begin{align} \label{variational}
   \sigma_k(\Omega)
= \min_{\substack{E \subset H^1(\Omega) \\ \dim(E)=k}}
   \; \max_{\substack{u \in E \\ u \ne 0}}
   \frac{\displaystyle\int_{\Omega} \|\nabla u\|^2\,dV +\int_{\partial B_r} \beta(x) u^2 dS}
        {\displaystyle\int_{\partial\Omega_{out}} u^2\,dS }.
\end{align}

\subsection{State of the Art}
 The Laplace eigenvalues on doubly connected domains have been extensively studied under various boundary conditions. To the best of our knowledge, P{\'o}lya~\cite{polya1960two} initiated the study of eigenvalue problem by deriving upper bounds for the first Dirichlet eigenvalue of planar ring-shaped domains. In recent years, significant progress has been made in obtaining eigenvalue bounds for mixed boundary value problems, including the mixed Dirichlet--Neumann problem~\cite{anoop2020reverse, anoop2021shape}, the mixed Steklov--Dirichlet problem~\cite{gavitone2024monotonicity, hong2020shape, GloriaPaoli2021C, sannipoli2025estimates}, the mixed Steklov--Neumann problem~\cite{basak2026sharp, arias2024applications, banuelos2010eigenvalue, hassannezhad2020eigenvalue}, and the mixed Robin--Neumann problem~\cite{CITO2026114086, cito2025stability, paoli2020sharp}.
Here, we focus on mixed Steklov-type problems on doubly connected bounded domain, for which several important results are known.
In~\cite{verma2020eigenvalue}, the authors proved that concentric annuli maximizes the first eigenvalue of the Steklov--Dirichlet problem among all annular domains of fixed volume. Using the spherical coordinate system, Ftouhi~\cite{ftouhi2022place} provided an alternative proof and extended the result to the classical Steklov problem.  Eigenvalue estimates for doubly connected star-shaped domains have also been explored in various geometric settings. {In~\cite{gavitone2023isoperimetric}, the authors derived upper bounds for the first Steklov--Dirichlet eigenvalue on doubly connected star-shaped domains and established the concentric annulus as the extremal domain in $\mathbb{R}^n$.} These results were later generalized to non-Euclidean space forms in~\cite{basak2023sharp}. Employing similar techniques, the authors in~\cite{gavitone2023steklov} established lower bounds for the first eigenvalue of the Steklov--Robin problem~\eqref{eqn: SD problem}.  


While these works primarily concern the first {nonzero} eigenvalue, more recent studies have explored isoperimetric bounds for higher Steklov-type eigenvalues on perforated domains.
In~\cite{basak2023sharp}, isoperimetric bounds for higher Steklov--Dirichlet eigenvalues were obtained. In particular, the authors showed that, among all doubly connected domains with symmetry of order $4$ and fixed volume, the corresponding concentric annulus provides an upper bound for the higher Steklov--Dirichlet eigenvalues. This results were later extended to the mixed Steklov--Neumann and Steklov eigenvalue problems under the same assumptions \cite{basak2026bounds}, where, in the Steklov–Neumann case, the inner ball can be replaced by any bounded domain having symmetry of order $4$. To best of our knowledge, the corresponding problem for the mixed Steklov--Robin eigenvalue problem remains open. Motivated by these developments, in the present article, we investigate isoperimetric upper bounds for higher mixed Steklov--Robin eigenvalues. Before proceeding, we introduce a symmetry notion that will be used to state the main result. 

\begin{definition}
    A domain  $W\subset \mathbb{R}^n$ is said to be symmetric of order $s$ with respect to the origin, if $R_{i,j}^{\frac{2\pi}{s}}(W)=W$ for all $1\leq i <j \leq n,$ where $R_{i,j}^{\frac{2\pi}{s}}$ denotes the anticlockwise rotation with respect to the origin by angle $\frac{2\pi}{s}$ in the coordinate plane $(x_i, x_j).$
\end{definition}

 Let $B_{R_1}$ be a ball in $\mathbb{R}^n$ of radius $R_1$ centered at the origin. Let $\Omega_{out} \subset \mathbb{R}^n$ be a bounded domain with Lipschitz boundary, having symmetry of order 4 centered at the origin such that $ \overline{ B_{R_1} }\subset \Omega_{out}$. Consider the Steklov--Robin eigenvalue problem \eqref{eqn: SD problem} on $\Omega=\Omega_\mathrm{out}\setminus \overline{B_{R_1}}.$  Under this setting, the following theorem holds.
\begin{theorem} \label{thm:isoperimetric}
    Let $\Omega$ be a bounded domain with Lipschitz boundary defined as above, and $\sigma_{k}$ be the $k$-th nonzero eigenvalue of \eqref{eqn: SD problem} on $\Omega$. Then for ${2} \leq k \leq n+1$, 
    \begin{equation}
        \sigma_{k}(\Omega) \leq \sigma_{k}(\Omega_0) = \sigma_2 (\Omega_0), 
    \end{equation} 
    where $\Omega_0 = B_{R_2} \backslash \overline{ B_{R_1}}$ and $B_{R_2}$ is a ball of radius $R_2$ centered at the origin such that $Vol(B_{R_2})= Vol(\Omega_{out})$ i.e, $Vol(\Omega)= Vol(\Omega_0)$.
    \end{theorem}

This article is organized as follows. In Section \ref{Pre}, we derive the eigenvalues and corresponding eigenfunctions on the annular domain and establish several integral inequalities. Section \ref{main} is devoted to the proof of Theorem \ref{thm:isoperimetric}. Finally, in Section \ref{counter}, we present examples of domains demonstrating that the symmetry assumption in Theorem \ref{thm:isoperimetric} cannot be dropped.



\section{Preliminaries}\label{Pre}
\subsection{The mixed  Steklov-Robin eigenvalues on annular domains.}

Let $\Omega_0 = B_{R_2} \backslash \overline{B_{R_1}}$, where $B_{R_1}$ and $B_{R_2}$ are concentric balls in $\mathbb{R}^{n}$ of radius $R_1$ and $R_2$ respectively, $0 < R_1 < R_2$. Without loss of generality, we assume that both balls $B_{R_1}$ and $B_{R_2}$ are centered at the origin. Assume that $\beta(x)=\beta>0$ is a positive constant on $\partial B_r$. Now, we find eigenvalues and corresponding eigenfunctions of the following Steklov-Robin eigenvalue problem on domain $\Omega_0$ 
and study some of their properties. 
\begin{equation}
    \begin{cases}
        \Delta u=0 & \text{in} \,\, \Omega_0=B_{R_2} \backslash \overline{B_{R_1}} ,\\
        \frac{\partial u}{\partial \nu}+\beta u=0 & \text{on}\,\, \partial B_{R_1},\\
        \frac{\partial u}{\partial \nu }=\sigma u & \text{on} \,\, \partial B_{R_2.}  \label{eq:Steklov-Robin(ball)}
    \end{cases}
\end{equation}
Let $u(r,\omega)=f(r)g(\omega)$ be a smooth function, where $f$ is a radial function defined on $[R_1, R_2]$ and $g$ is an eigenfunction of $\Delta_{S^{n-1}}$ corresponding to the eigenvalue 
$l(l+n-2)$. 
Now,
\begin{align*}
     \Delta u(r, \omega) &= g(\omega)\left(-f''(r)- \frac{n-1}{r} f'(r) \right) + \frac{f(r)}{r^2} \Delta_{
     S
     ^{n-1}} g(\omega)\\
     &= g(\omega) \left(-f''(r)- \frac{n-1}{r} f'(r) + \frac{f(r)}{r^2} l(l+n-2) \right).
\end{align*}
If $u$ is a solution of (\ref{eq:Steklov-Robin(ball)}), then $f$ satisfies the following ordinary differential equation and boundary constraints.
\begin{eqnarray} \label{ODE for f}
\begin{cases}
    \begin{array}{ll}
        -f''(r)-\frac{n-1}{r}f'(r)+ \frac{l(l+n-2)}{r^2}f(r)=0 ~\mbox{ for } ~ r \in (R_1,R_2), \\
    -f'(R_1)+\beta f(R_1)=0, \,\, f'(R_2)=\sigma f(R_2).
    \end{array}
    \end{cases}
\end{eqnarray}
For $l=0,$ the following first eigenvalue and corresponding eigenfunction of \eqref{eq:Steklov-Robin(ball)} are obtained in \cite{gavitone2023steklov}.
\begin{align*}
    \sigma_{1}(\Omega_0)=\begin{cases}
    \frac{1}{\frac{R_2}{\beta R_1}+R_2 \log(\frac{R_2}{R_1})}   \quad n=2,\\
        \frac{\beta R_2^{1-n}}{\beta R_2^{2-n}-\beta R_1^{2-n}-(n-2)R_1^{1-n}} \quad n\geq 3.
    \end{cases}
\end{align*}

\begin{align}\label{f_0 eigenfunction}
    f_0(r)=\begin{cases}
        \log(\frac{r}{R_1})+\frac{1}{\beta R_1}  \quad n=2,\\
        \frac{1}{R_1^{n-2}}-\frac{1}{r^{n-2}}+\frac{n-2}{\beta}\frac{1}{R_2R_1^{n-2}} \quad n\ge 3.
    \end{cases}
\end{align}
Now we compute eigenvalues and eigenfunctions for $l\geq 1.$
Solving by ordinary differential equation \eqref{ODE for f}
we get 
\begin{align}\label{solu ODE}
    f_l(r)=C_1r^l+C_2r^{-l-n+2}
\end{align}
Applying the boundary conditions, we obtain the following system of equations for the unknowns $C_1$ and $C_2$.
\begin{align} \label{System of linear}
    \begin{cases}
        C_1(-R_1^{l-1}l+\beta R_1^l)+C_2(\beta R_1^{-l-n-2}-(-l-n+2)R_1^{-l-n+1})=0\\
        C_1(-lR_2^{l-1}+R_2^l)+C_2(\sigma R_2^{-l-n+2}-(-l-n+2)R_2^{-l-n+1})=0.
    \end{cases}
\end{align}
Since the system is homogeneous, it admits a nontrivial solution if and only if
\begin{align*}
\det
\begin{pmatrix}
-R_1^{l-1}l+\beta R_1^l & \beta R_1^{-l-n+2}-(-l-n+2)R_1^{-l-n+1}\\
-lR_2^{l-1}+\sigma R_2^l & \sigma R_2^{-l-n+2}-(-l-n+2)R_2^{-l-n+1}
\end{pmatrix}
=0.
\end{align*}
From this, we get
 \begin{align*}
     \sigma_{(l)}= \frac{l R_1^{-l-n+1}R_2^{l-1}(\beta R_1+l+n-2)-R_1^{l-1}R_2^{1-n-l}(\beta R_1-l)(l+n-2)}{R_1^{l-1} R_2^{2-l-n}(\beta R_1 -l)-R_1^{-l-n-1}R_2^l(\beta R_1 +l+n-2)}, \quad l\geq 1.
 \end{align*} 
For a given value of $\sigma,$ constants $ C_1$ and $C_2$ are linearly dependent. Choosing $C_1=1,$ the first equation in \eqref{System of linear} yields
\begin{align*}
    C_2=-\frac{-lR_1^{l-1}+\beta R_1^l}{\beta R_1^{-l-n+2}-(-l-n+2)R_1^{-l-n+2}}.
\end{align*}
 Hence, inserting $C_1$ and $C_2$ in \eqref{solu ODE}, we have 
 
 \begin{align} \label{eigenfunction}
     f_l(r)=r^l-\frac{-lR_1^{l-1}+\beta R_1^l}{\beta R_1^{2-l-n}-(2-l-n)R_1^{1-l-n}}r^{2-l-n}, \quad l\geq 1.
 \end{align}
 Here, $\sigma_{2}(\Omega_0)=\sigma_{3}(\Omega_0)=\cdots=\sigma_{n+1}(\Omega_0)=\sigma_{(1)}$ denotes the second  Steklov--Robin eigenvalue on $\Omega_0$, counted without multiplicity. The corresponding eigenfunctions are of the form
$
f_1(r)\frac{x_i}{r},$ where $(x_1,x_2,\ldots,x_n)$ are the Cartesian coordinates in $\mathbb{R}^n$.  

 \begin{lemma} \label{increasing decreasin lmm}
     Let $f_1:[R_1, \infty ) \to \mathbb{R}$ be the function defined as in \eqref{eigenfunction} for $l=1.$ Define $F,G:[R_1, \infty) \to \mathbb{R}$ as 
     \begin{align*}
         F(r)=(f_1'(r))^2+\frac{n-1}{r^2}f_1^2(r) \quad \textit{and} \quad G(r)=2f_1(r)f_1'(r)+\frac{n-1}{r}f_1^2(r).
     \end{align*}
     Then $F$ is a nonnegative decreasing function of $r$ and $G$ is a nonnegative increasing function of r.
 \end{lemma}
 \begin{proof}
     Recall that $f_1(r)= r+ \frac{R_1^n(1-\beta R_1)}{\beta R_1 +n-1}\frac{1}{r^{n-1}},$  then $f_1'(r)= 1-\frac{R_1^n(1-\beta R_1)}{\beta R_1 +n-1}\frac{n-1}{r^{n}}$. Substituting these values, we get 
     \begin{align*}
         &F(r)=n+ \frac{R_1^{2n}(1-\beta R_1)^2(n^2-n)}{(\beta R_1+n-1)^2}\frac{1}{r^{2n}} ,\quad \textit{and}\\
         &G(r)=(n+1)r+2 \frac{R_1^n(1-\beta R_1)}{(\beta R_1 +n-1)}\frac{1}{r^{n-1}}- \frac{R_1^{2n}(1-\beta R_1)^2(n-1)}{(\beta R_1 +n-1)^2}\frac{1}{r^{2n-1}}.
     \end{align*}
     By differentiating these functions, we have for all $r\in [R_1, \infty),$
     \begin{align*}
         &F'(r)= - \frac{R_1^{2n}(1-\beta R_1)^2(n^2-n)}{(\beta R_1+n-1)^2}\frac{2n}{r^{2n-1}}\leq 0, \quad\textit{and} \\
        & G'(r)=2+(n-1)\left(\frac{R_1^{n}(1-\beta R_1)}{(\beta R_1+n-1)}\frac{1}{r^{n,}}-1\right)^2 + \frac{2R_1^{2n}(1-\beta R_1)^2(n-1)}{(\beta R_1+n-1)^2}\frac{1}{r^{2n,}}\geq 0.
     \end{align*}
     Therefore, $F(r)$ is a decreasing function and $G(r)$ is an increasing function on $[R_1, \infty).$ 

     Further, from the expression of $F(r),$ it is clear that $F(r)\geq 0$ for all $r\geq R_1.$ We now prove that $G(r)\geq 0$ on interval $[R_1, \infty).$
     We compute 
     \begin{align*}
         G(R_1)&= (n+1)R_1 +\frac{R_1(1-\beta R_1)}{(\beta R_1 +n-1)}- \frac{R_1(1-\beta R_1)^2(n-1)}{(\beta R_1 +n-1)^2}\\
         &= (n+1)R_1 + \frac{R_1(1-\beta R_1)}{(\beta R_1 +n-1)^2}\left( (n+1)\beta R_1 +(n-1) \right)\\
         &=\frac{R_1}{(\beta R_1 +n-1)^2} (\beta R_1(n^2+1)+(n-1)^2(n+1)+(n-1))\geq 0.
     \end{align*}
     Since $G(r)$ is increasing on $[R_1, \infty),$ it follows that $G(r)\geq G(R_1)\geq0,$ on $[R_1, \infty).$
 \end{proof}

\subsection{Some integral inequalities}
 In this subsection, we establish several integral inequalities that will be useful in the proof of our main results.
 \begin{lemma} \label{lem:integral2}
    Let $F$ be a decreasing radial function defined as in Lemma \ref{increasing decreasin lmm} and $\Omega, \Omega_0$ be defined as above. Then the following inequality holds.
    \begin{equation}
      \int_\Omega F(r) dV \leq \int_{\Omega_0} F(r) dV.
    \end{equation}
\end{lemma}
\begin{proof}
Let $F:\mathbb{R}^{n}\setminus B_{R_1}\to\mathbb{R}$ be a nonnegative decreasing radial function. Define the extension
\[
\widetilde{F}:\mathbb{R}^{n}\to\mathbb{R}
\]
by
\[
\widetilde{F}(x)=
\begin{cases}
F(R_1), & x\in B_{R_1},\\[1mm]
F(|x|), & x\in \mathbb{R}^{n}\setminus B_{R_1}.
\end{cases}
\]
Then $\widetilde{F}$ is a nonnegative decreasing radial function on $\mathbb{R}^{n}$. Moreover,
\[
\int_{\Omega}\!F(|x|)\,dx
=
\int_{\Omega_{\mathrm{out}}}\!\widetilde{F}(x)\,dx
-
F(R_1)|B_{R_1}|.
\]
Since
\[
\int_{\Omega_{\mathrm{out}}}\widetilde{F}(x)\,dx
=
\int_{\mathbb{R}^{n}}
\widetilde{F}(x)\chi_{\Omega_{\mathrm{out}}}(x)\,dx,
\]
the Hardy--Littlewood rearrangement inequality yields
\[
\int_{\mathbb{R}^{n}}
\widetilde{F}(x)\chi_{\Omega_{\mathrm{out}}}(x)\,dx
\le
\int_{\mathbb{R}^{n}}
\widetilde{F}^{*}(x)\chi_{\Omega_{\mathrm{out}}}^{*}(x)\,dx.
\]
Since $\widetilde{F}$ is a nonnegative radial decreasing function, it coincides with its symmetric decreasing rearrangement, namely
\[
\widetilde{F}^{*}=\widetilde{F}.
\]
Furthermore,
\[
\chi_{\Omega_{\mathrm{out}}}^{*}
=
\chi_{\Omega_{\mathrm{out}}^{*}}
=
\chi_{B_{R_2}},
\]
where $B_{R_2}$ is the ball satisfying
\[
|B_{R_2}|=|\Omega_{\mathrm{out}}|.
\]
Hence,
\[
\int_{\mathbb{R}^{n}}
\widetilde{F}^{*}(x)\chi_{\Omega_{\mathrm{out}}}^{*}(x)\,dx
=
\int_{B_{R_2}}\widetilde{F}(x)\,dx.
\]
Combining the above identities, we obtain
\[
\int_{\Omega}F(|x|)\,dx
\le
\int_{B_{R_2}}\widetilde{F}(x)\,dx
-
F(R_1)|B_{R_1}|.
\]
Finally, by the definition of $\widetilde{F}$,
\[
\int_{B_{R_2}}\widetilde{F}(x)\,dx
-
F(R_1)|B_{R_1}|
=
\int_{B_{R_2}\setminus B_{R_1}}F(|x|)\,dx.
\]
Therefore,
\[
\int_{\Omega}F(|x|)\,dx
\le
\int_{B_{R_2}\setminus B_{R_1}}F(|x|)\,dx,
\]
which completes the proof.
\end{proof}

\begin{lemma}\label{integral6}
   Let $f_{1}(r), \Omega_{out}$ and $B_{R_2}$ be defined as above, and $\partial \Omega_{out}, \partial B_{R_2}$ denote boundaries of $\Omega_{out}, B_{R_2},$ respectively. Then the following inequality holds.
   \begin{equation} \label{inequality 5}
      \int_{\partial {\Omega_{out}}} f_{1}^2(r) dS \geq  \int_{\partial {B_{R_2}}} f_{1}^2(r) dS.
   \end{equation}
\end{lemma} 

\begin{proof}
    Using the positivity and monotonicity of $G$, the desired inequality can be established by an argument similar to that in \cite{basak2026bounds}.
\end{proof}
We will use the following proposition from \cite{basak2023sharp} to conclude some integral identities related to the function $f_{1}(r)$. 
\begin{proposition} \label{prop:integral expression}
 Let $h : (0, \infty) \rightarrow \mathbb{R}$ be a smooth positive radial function of $r$. Let $W$  be a bounded domain in $\mathbb{R}^n$ with Lipschitz boundary $\partial W$.
 \begin{enumerate}
     \item If $W$ is symmetric of order $2$ with respect to the origin, then for each $i=1,2,\ldots n$, we have
     \begin{enumerate}
         \item $\displaystyle \int_{x \in W}   h(r) x_i\, dV = 0$,  \label{eqn:integral 1}~~~
       \item $\displaystyle \int_{x \in \partial W}   h(r) x_i\, dS = 0$.  \label{eqn: integral 3}
        \end{enumerate}
\item If $W$ is  symmetric of order $4$ with respect to the origin, then for each $i, j=1,2,\dots n$, $i\neq j$,  we have
\begin{enumerate}
     \item $\displaystyle \int_{ x \in W} h(r) x_i x_j \, dV = 0$, \label{eqn:integral 2}
      \item $\displaystyle \int_{x \in \partial W}  h(r) x_i x_j \,dS = 0$.  \label{eqn: integral 4}
        \end{enumerate}
 \end{enumerate}
 Here $(x_1, x_2, \ldots, x_n)$ represents Cartesian coordinates of a point $x \in \mathbb{R}^{n}$ with respect to the origin.
\end{proposition}
Let $E_1, E_2,\dots, E_n$ be the standard orthonormal basis of $\mathbb{R}^n$. Then for the eigenfunction $f_{1}$ given in \eqref{eigenfunction}, we get
\begin{align*} 
\Bigg \langle  \nabla \left( \frac{f_{1}(r)}{r}x_i \right), E_j \Bigg \rangle = \frac{\partial}{\partial x_j} \left( \frac{f_{1}(r)}{r}x_i \right) = 
\begin{cases}
   \frac{f_{1}'(r)}{r^2} x_j \ x_i- \frac{f_{1}(r)}{r^3} x_j \ x_i, &  \text{ for } j \neq i, \\[2mm]
    \frac{f_{1}'(r)}{r^2} x_i^2+ \frac{f_{1}(r)}{r^3}(r^2-x_i^2), &  \text{ for } j = i.
\end{cases}
\end{align*}
Consequently, we get,
\begin{align} \label{eq: inner product of gradient1}
\Bigg \langle  \nabla \left( \frac{f_{1}(r)}{r}x_i \right), \nabla \left( \frac{f_{1}(r)}{r}x_j \right) \Bigg \rangle =  
\begin{cases}
   \left( \frac{(f_{1}'(r))^2}{r^2} - \frac{(f_{1}(r))^2}{r^4}\right)x_ix_j, &  \text{ for } j \neq i, \\[2mm]
    \left( \frac{(f_{1}'(r))^2}{r^2}x_i^2 - \frac{f_{1}^2(r)}{r^4}x_i^2  + \frac{f_{1}^2(r)}{r^2}\right), &  \text{ for } j = i.\\
\end{cases}
\end{align}

\begin{align} \label{eq: inner product of gradient2}
 \mbox{ And, }\qquad \qquad\Bigg \langle \nabla\left(f_{1}(r)\right),\nabla \left(\frac{f_{1}(r)}{r}x_j\right)  \Bigg \rangle = \frac{\left(f_{1}'(r)\right)^2}{r}x_j.  \qquad \qquad \quad \quad \quad 
\end{align}
Using these expressions and Proposition \ref{prop:integral expression}, we conclude the following corollary.
\begin{cor} \label{cor: f_{1,1}}
    Let $\Omega$ and $f_{1}$ be defined as in the beginning of this section. Then for each $i,j=1,2,\dots n$, $i\neq j$ we have 
    \begin{enumerate}
        \item $\displaystyle \int_{x\in \partial \Omega}f_{1}(r) \frac{f_{1}(r)}{r}x_i\, dS = 0,$ \label{proof (i)}\\
         \item $\displaystyle \int_{x \in \partial \Omega}  \frac{f_{1}(r)}{r}x_i. \frac{f_{1}(r)}{r}x_j \,dS=0,$
        \item $\displaystyle \int_{ x \in \Omega} \left< \nabla f_{1}(r), \nabla \left(\frac{f_{1}(r)}{r}x_i \right)\right> dV=0,$ 
        \item $\displaystyle \int_{x \in \Omega} \left< \nabla \left(\frac{f_{1}(r)}{r} x_i \right) , \nabla \left(\frac{f_{1}(r)}{r}x_j \right)\right> dV=0$. 
    \end{enumerate}
\end{cor}
\begin{lemma} \label{lem:integral1}
    Let $W \subset \mathbb{R}^n$ be an open, bounded domain with Lipschitz boundary having symmetry of order $4$ with respect to the origin. Let $\Phi(r)$ be a positive radial function on $\mathbb{R}^n$. Then, there exists a constant $A>0$ such that
    \begin{equation}
         \int_{W} \Phi(r) x_i^2\,\, dV= A\,\, \text{for all}\, \,i\in \{1,2 \dots n\}.
    \end{equation}
\end{lemma}
For a detailed proof of Proposition \ref{prop:integral expression}, Corollary \ref{cor: f_{1,1}} and Lemma \ref{lem:integral1}, see (\cite{basak2023sharp}).

\begin{remark}
To prove Corollary~\ref{cor: f_{1,1}} and Lemma~\ref{lem:integral1}, we use `the symmetry of order~$4$' condition, which in particular yields an orthogonality condition of the test functions. Moreover, Corollary~\ref{cor: f_{1,1}} is used directly in the proof of Theorem~\ref{thm:isoperimetric}.
\end{remark}

\section{Proof of the main result}\label{main}
Now we provide proof of the main theorem.

  \noindent \textbf{Proof of Theorem \ref{thm:isoperimetric}:} 
        Since $\sigma_k(\Omega_0)=\sigma_2(\Omega_0), 2\leq k \leq n+1 $, it is enough to prove that $\sigma_{n+1}(\Omega) \leq \sigma_{n+}(\Omega_0)=\sigma_1(\Omega_0)$. Consider the following $(n+1)$ dimensional subspace of $H^1(\Omega)$,
        \begin{align*}
            E=span\{ f_{1}(r), \frac{f_{1}(r)}{r}x_1, \dots \frac{f_{1}(r)}{r}x_n \}, 
        \end{align*}
        where $f_{1}(r)$ is defined in (\ref{eigenfunction}) with $l=1$. For any $u \in E \backslash \{0\}$, there exist $c_0, c_1, \dots c_n \in \mathbb{R}$ not simultaneously equal to zero, such that 
        \begin{align*}
            u= c_0 f_{1}(r) + c_1 \frac{f_{1}(r)}{r}x_1 + \dots + c_n\frac{f_{1}(r)}{r}x_n.
        \end{align*}
        Then by using Corollary \ref{cor: f_{1,1}}, we get 
        \begin{align}
            \frac{\displaystyle \int_\Omega \| \nabla u \|^2 dV}{\displaystyle \int_{\partial {\Omega_{out}}} u^2 dS}= \frac{c_0^2 \displaystyle \int_\Omega \|\nabla f_{1}(r)\|^2 \, dV+ \displaystyle \sum_{i=1}^nc_i^2 \displaystyle \int_\Omega \bigg \| \nabla \left( \frac{f_{1}(r)}{r}  x_i \right) \bigg \|^2 \, dV}{c_0^2 \displaystyle \int_{\partial {\Omega_{out}}}f_{1}^2(r) \, dS + \displaystyle \sum_{i=1}^n c_i^2 \displaystyle \int_{\partial {\Omega_{out}}} \frac{f_{1}^2(r)}{r^2}x_i^2 \, dS}. \label{equality}
        \end{align}
        According to Lemma \ref{lem:integral1}, there are constants $A_1, A_2 > 0 $ such that for all natural numbers $ 1\leq i \leq n$,
         \begin{align*}
       \int_{\partial {\Omega_{out}}}\left( \frac{f_{1}(r)}{r}x_i \right)^2 dS & =\int_{\partial {\Omega_{out}}} \frac{f_{1}^2(r)}{r^2}x_i^2 dS = A_1,\\
       \int_\Omega \bigg \| \nabla \left( \frac{f_{1}(r)}{r} x_i \right) \bigg \|^2 dV & = \int_\Omega \left( \frac{(f'_{1}(r))^2}{r^2} x_i^2 - \frac{f_{1}^2(r)}{r^4}x_i^2 + \frac{f_{1}^2(r)}{r^2} \right) dV = A_2.
   \end{align*} 
   Therefore
   $$ n\,A_1 = \sum_{i=1}^n \int_{\partial {\Omega_{out}}} \left( \frac{f_{1}(r)}{r}x_i \right)^2 \, dS = \int_{\partial {\Omega_{out}}} f_{1}^2(r) \, dS, $$ and
 $$n\,A_2 = \sum_{i=1}^n  \int_\Omega \left(\frac{(f'_{1}(r))^2}{r^2} x_i^2-\frac{f_{1}^2(r)}{r^4}x_i^2 +\frac{f_{1}^2(r)}{r^2} \right) dV =\int_\Omega \left((f'_{1}(r))^2 + \frac{(n-1)}{r^2}f_{1}^2(r)\right)\, dV.$$
 Thus for all natural numbers $ 1\leq i\leq n,$ we have 
  \begin{equation} \label{sum 1} 
       \displaystyle \int_{\partial {\Omega_{out}}} \left( \frac{f_{1}(r)}{r}x_i \right)^2 \, dS  = A_1 = \frac{1}{n} \displaystyle \int_{\partial {\Omega_{out}}} f_{1}^2(r) \, dS.
        \end{equation}
         \begin{equation} \label{sum 1.5} 
        \displaystyle \int_\Omega \bigg \| \nabla \left( \frac{f_{1}(r) x_i}{r}\right) \bigg \|^2 \, dV  = A_2 = \frac{1}{n} \displaystyle \int_\Omega \left( (f'_{1}(r))^2 + \frac{(n-1)}{r^2} f_{1}^2(r)\right) \, dV.
         \end{equation}
         Now, from (\ref{equality}), \eqref{sum 1} and  \eqref{sum 1.5}, we get 
   \begin{align}
       \frac{\displaystyle \int_\Omega \| \nabla u \|^2 dV}{\displaystyle \int_{\partial {\Omega_{out}}} u^2 \, dS} = \frac{c_0^2 \displaystyle \int_\Omega \|\nabla f_{1}(r)\|^2 \, dV + A_2 \displaystyle \sum_{i=1}^n c_i^2}{c_0^2 \displaystyle \int_{\partial {\Omega_{out}}}f_{1}^2(r) \, dS + A_1 \displaystyle \sum_{i=1}^n c_i^2 }
        \leq \text{max} \Bigg \{\frac{\displaystyle \int_\Omega \|\nabla f_{1}(r)\|^2 \, dV}{\displaystyle \int_{\partial {\Omega_{out}}}f_{1}^2(r) \, dS}, \frac{A_2}{A_1} \Bigg \}. \label{gradient inequality 1}
   \end{align}
   Since 
\begin{align*}
    \frac{A_2}{A_1} & = \frac{\displaystyle \int_\Omega \left( (f'_{1}(r))^2 + \frac{(n-1)}{r^2} f_{1}^2(r)\right) dV}{\displaystyle \int_{\partial {\Omega_{out}}} f_{1}^2(r) \, dS} 
                     \geq \frac{\displaystyle \int_\Omega (f'_{1}(r))^2 \, dV}{\displaystyle \int_{\partial {\Omega_{out}}} f_{1}^2(r) \, dS} 
                     = \frac{\displaystyle \int_\Omega \|\nabla f_{1}(r)\|^2 \, dV}{\displaystyle \int_{\partial {\Omega_{out}}}f_{1}^2(r) \, dS}.
\end{align*}
   Then from the inequality (\ref{gradient inequality 1}) we get
   \begin{align}
       \frac{\displaystyle \int_\Omega \| \nabla u \|^2 dV}{\displaystyle \int_{\partial {\Omega_{out}}} u^2 \, dS} \leq  \frac{A_2}{A_1} =\frac{\displaystyle \int_\Omega \left( (f'_{1}(r))^2 + \frac{(n-1)}{r^2} f_{1}^2(r)\right) \, dV }{\displaystyle \int_{\partial {\Omega{out}}} f_{1}^2(r) \, dS}.\label{1st ratio} 
    \end{align} 
   Moreover, as the ball $B_{R_1}$ is symmetry of order $4$, it follows from Lemma \ref{lem:integral1} that
\begin{align}
\frac{\displaystyle \int_{\partial B_{R_1}} \beta u^2\,dS}
{\displaystyle \int_{\partial\Omega_{out}} u^2\,dS}
&=
\frac{\displaystyle \beta\left(
c_0^2\int_{\partial B_{R_1}} f_1^2(r)\,dS
+\sum_{i=1}^n c_i^2
\int_{\partial B_{R_1}}\frac{f_1^2(r)}{r}x_i^2\,dS
\right)}
{\displaystyle
c_0^2\int_{\partial\Omega_{\mathrm{out}}} f_1^2(r)\,dS
+\sum_{i=1}^n c_i^2
\int_{\partial\Omega_{\mathrm{out}}}\frac{f_1^2(r)}{r}x_i^2\,dS} \nonumber\\
&=
\frac{\displaystyle
\beta\left(c_0^2+\frac{1}{n}\sum_{i=1}^n c_i^2\right)
\int_{\partial B_{R_1}} f_1^2(r)\,dS}
{\displaystyle
\left(c_0^2+\frac{1}{n}\sum_{i=1}^n c_i^2\right)
\int_{\partial\Omega_{\mathrm{out}}} f_1^2(r)\,dS} \nonumber \\
&=
\frac{\beta \displaystyle\int_{\partial B_{R_1}} f_1^2(r)\,dS}
{\displaystyle\int_{\partial\Omega_{\mathrm{out}}} f_1^2(r)\,dS}.\label{2nd ratio}
\end{align}

Adding \eqref{1st ratio} and \eqref{2nd ratio}, we consequently obtain
   \begin{align}\label{gradient inequality}
       \frac{\displaystyle \int_\Omega \| \nabla u \|^2 dV +\beta \int_{\partial B_{R_1}}  u^2 \,dS}{\displaystyle \int_{\partial {\Omega_{out}}} u^2 \, dS} \leq \frac{\displaystyle \int_\Omega \left( (f'_{1}(r))^2 + \frac{(n-1)}{r^2} f_{1}^2(r)\right) \, dV +  \beta \int_{\partial {B_{R_1}}}  f_1^2(r)\, dS }{\displaystyle \int_{\partial {\Omega_{out}}} f_{1}^2(r) \, dS}. 
   \end{align} 
    Next, using the Lemma \ref{lem:integral2} and  \ref{integral6}, we get  
\begin{align*}
       \frac{ \displaystyle \int_\Omega \left( (f_{1}'(r))^2 + \frac{(n-1)}{r^2} f_{1}^2(r)\right) \, dV + \displaystyle \beta \int_{\partial {B_{R_1}}}  f_{1}^2(r) \, dS }{\displaystyle \int_{\partial {\Omega_{out}}} f_{1}^2(r) \, dS}\\ 
       \leq \frac{\displaystyle \int_{\Omega_{0}} \left( (f_{1}'(r))^2 + \frac{(n-1)}{r^2} f_{1}^2(r)\right ) \, dV + \displaystyle \beta\int_{\partial {B_{R_1}}}  f_{1}^2(r) \, dS }{\displaystyle \int_{\partial B_{R_2}}f_{1}^2(r) \, dS}.
\end{align*}
Therefore, from the variational characterization (\ref{variational}) and inequality (\ref{gradient inequality}), we conclude
   \begin{align*}
      \sigma_{n+1}(\Omega) \leq \max_{u(\neq 0) \in E} \frac{\displaystyle \int_\Omega \| \nabla u \|^2 \, dV + \beta \int_{{\partial B_{R_1}}} u^2 dS}{\displaystyle \int_{\partial {\Omega_{out}}} u^2 \, dS} \leq \sigma_2(\Omega_0).
   \end{align*}

   \begin{remark} For the function $f_0$ defined in \ref{f_0 eigenfunction}, the function
$
\hat G(r):=2f_0(r)f_0'(r)-\frac{n-1}{r}f_0^2(r)
$ is not monotone increasing in general.
Hence, under the imposed symmetry assumptions on domain $\Omega_{out}$, inequality \eqref{integral6} cannot be established for the function $f_0$ using our current approach. Therefore, the method presented above cannot be directly applied to obtain an isoperimetric bound for the first Steklov--Robin eigenvalue on $\Omega$.
   \end{remark}

\section{Counter examples Based on Numerical evidence}\label{counter}

In this section, we show that the symmetry of order $4$ assumption in Theorem~\ref{thm:isoperimetric} is essential and, in general, cannot be removed. 

\begin{figure}[h]
\centering
\subfloat[The concentric annulus $\Omega_0$]{
  \includegraphics[width=0.39\textwidth]{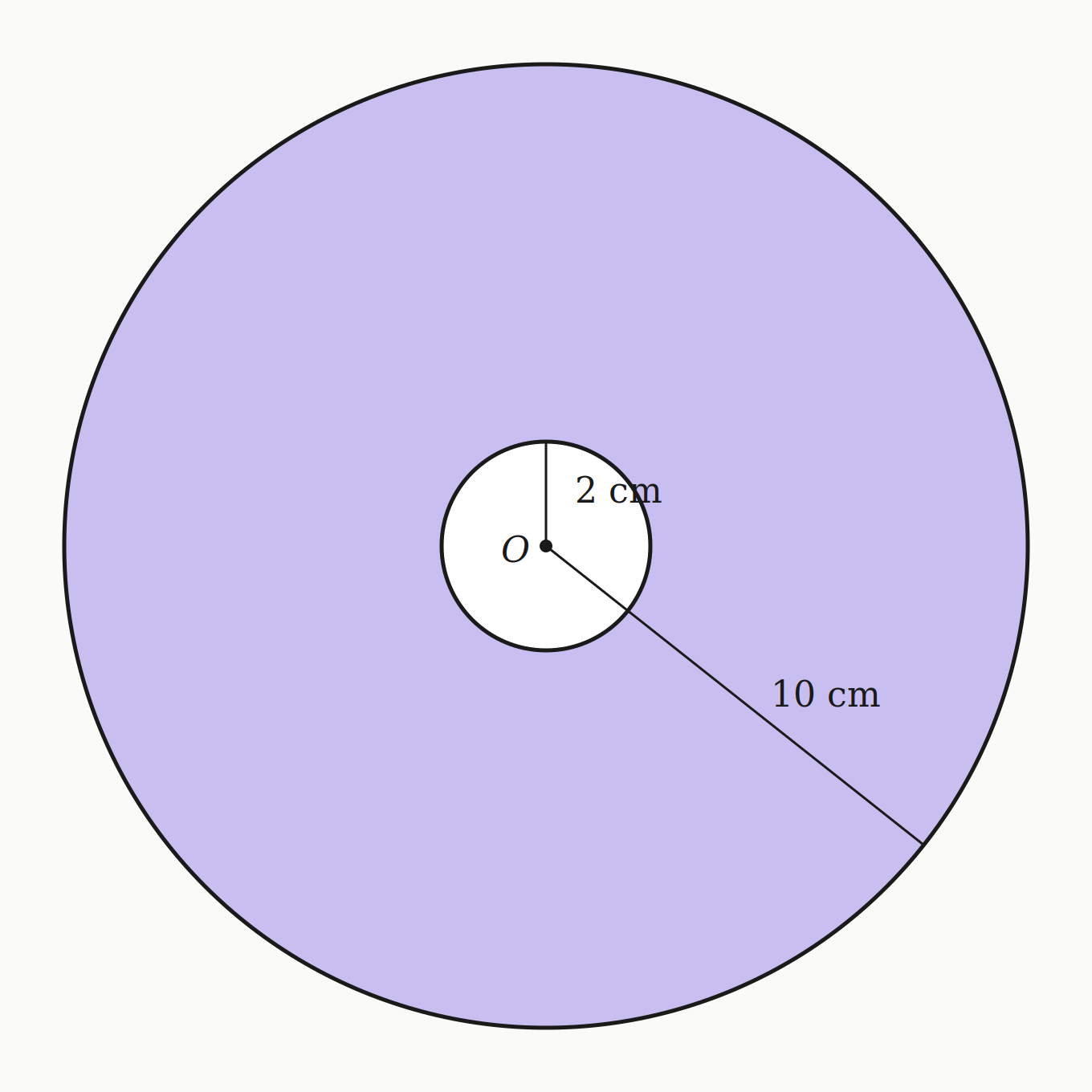}
}
\subfloat[$\Omega_2$ with symmetry of order $2$]{
  \includegraphics[width=0.30\textwidth]{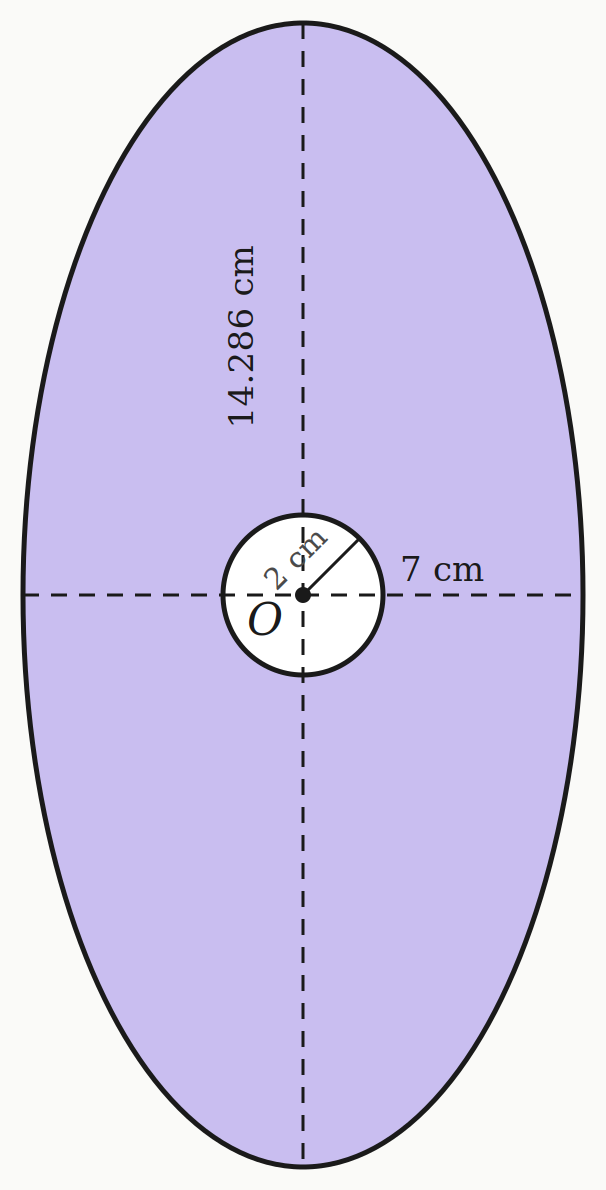}
}
\subfloat[$\Omega_3$ without any symmetry]{
  \includegraphics[width=0.30\textwidth]{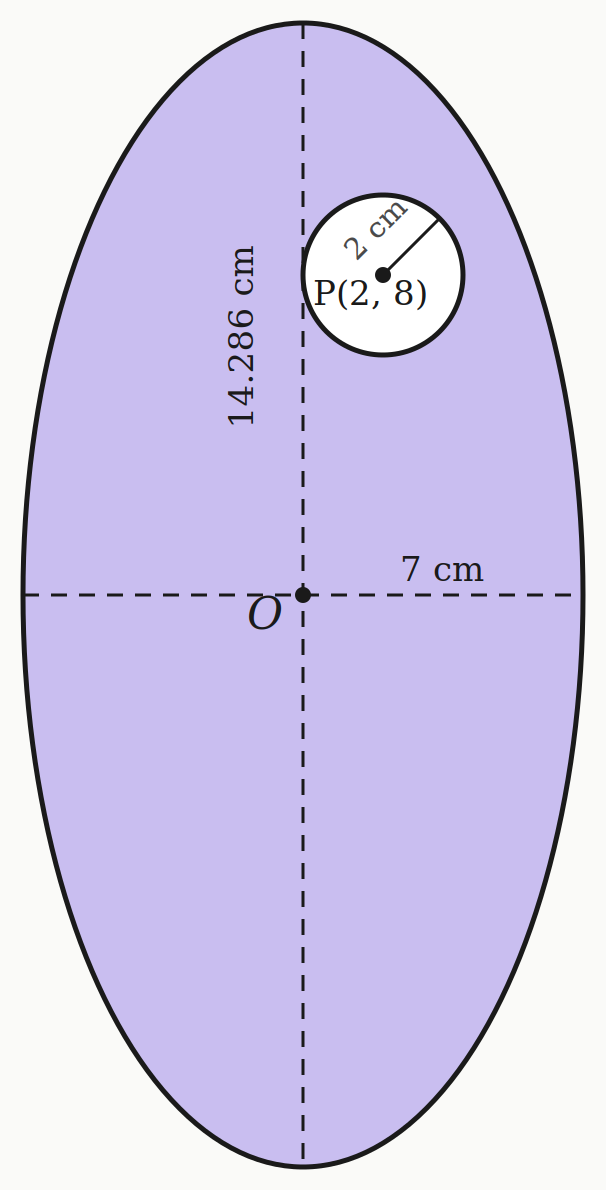}
}

\caption{Comparison of the eigenvalues $\sigma_2(\Omega_i)$ and $\sigma_2(\Omega_i)$ ($i=1,2$) with $\sigma_2(\Omega_0)=\sigma_3(\Omega_0)$, where $\Omega_0$ is the concentric annulus.}
\label{counterexamples}
\end{figure}
Specifically, we consider two doubly connected planar domains, denoted by $\Omega_1$ and  $\Omega_2$, all having the same area but lacking symmetry of order $4$. These domains are shown in Figure~\ref{counterexamples}.
Each domain is of the form
\[
\Omega_1=\tilde{E}\setminus\overline{B_2(O)} \quad \text{and} \quad \Omega_2=\tilde{E}\setminus \overline{B_2(P)},
\]
where $B_2(O)$ and $B_2(P)$ denote circular holes of radius $2$ cm centered at $O$ and $P$, respectively. Here, $\tilde{E}$ denotes the vertical ellipse with major and minor axes of lengths $14.286$ cm and $7$ cm, respectively. More precisely, $\Omega_1$ has symmetry of order $2$, while $\Omega_2$ has no symmetry.
Using \texttt{FreeFEM++}, for $\beta=1,$ we numerically compute the eigenvalues and obtain that
\[
\sigma_3(\Omega_1)>\sigma_3(\Omega_0)
\]
where $\Omega_0$ denotes the concentric annulus with the same area. Consequently, the conclusion of Theorem~\ref{thm:isoperimetric} fails for the $(n+1)$-th eigenvalue, even for domains possessing some symmetry but lacking symmetry of order $4$.\\
The domain $\Omega_2$, which has no symmetry, satisfies
\[
\sigma_2(\Omega_2)>\sigma_2(\Omega_0)
\quad\text{and}\quad
\sigma_3(\Omega_2)>\sigma_2(\Omega_0).
\]
Therefore, the conclusion of Theorem~\ref{thm:isoperimetric} fails for both $\sigma_2$ and $\sigma_3$ in the absence of the symmetry of order $4$ assumption.

The computed Steklov--Robin eigenvalues for the three domains, with $\beta=1$, are listed in Table~\ref{tab:mytable}.

\begin{table}[h]
\centering
\begin{adjustbox}{width=\textwidth}
\begin{tabular}{|c|c|c|c|c|}
\hline
Steklov--Dirichlet
Eigenvalue
& $\Omega_0$
& $\Omega_1$
& $\Omega_2$\\
\hline
Description
& Concentric annulus
& Concentric ellipse--ball
& Non-concentric ellipse--ball\\
\hline
$\sigma_2$
& 0.102707
& 0.054528
& 0.110097\\
\hline
$\sigma_3$
& 0.102707
& 0.161745
& 0.156683\\
\hline
\end{tabular}
\end{adjustbox}
\caption{Comparison of $\sigma_2(\Omega_0)=\sigma_3(\Omega_0)$ with $\sigma_2(\Omega_i)$ and $\sigma_3(\Omega_i)$, $i=1,2$.}
\label{tab:mytable}
\end{table}

\textbf{Acknowledgment :} S. Basak is supported by the University Grants Commission, India. We include proof of Lemma \ref{lem:integral2}, which was communicated to us by a reviewer during the review process of another paper \cite{basak2026bounds}, we thank the reviewer for providing this alternative argument. 

\bibliographystyle{plain}
\bibliography{ref}

@article{basak2023sharp,
  title={Sharp bounds for higher {S}teklov-{D}irichlet eigenvalues on domains with spherical holes},
  author={Basak, S. and Chorwadwala, A. and Verma, S.},
  journal={Canadian Mathematical Bulletin},
  pages={1--20},
  year={2023},
  publisher={Canadian Mathematical Society}
}

@article{polya1960two,
  title={Two more inequalities between physical and geometrical quantities},
  author={P{\'o}lya, G.},
  journal={J. Indian Math. Soc.(NS)},
  volume={24},
  number={1961},
  pages={413--419},
  year={1960}
}

@article{verma2020eigenvalue,
  title={On eigenvalue problems related to the {L}aplacian in a class of doubly connected domains},
  author={Verma, S. and Santhanam, G.},
  journal={Monatshefte f{\"u}r Mathematik},
  volume={193},
  number={4},
  pages={879--899},
  year={2020},
  publisher={Springer}
}

@article{ftouhi2022place,
  title={Where to place a spherical obstacle so as to maximize the first nonzero {S}teklov eigenvalue},
  author={Ftouhi, I.},
  journal={ESAIM: Control, Optimisation and Calculus of Variations},
  volume={28},
  pages={6},
  year={2022},
  publisher={EDP Sciences}
}

@article{gavitone2023isoperimetric,
  title={An isoperimetric inequality for the first {S}teklov--{D}irichlet {L}aplacian eigenvalue of convex sets with a spherical hole},
  author={Gavitone, N. and Paoli, G. and Piscitelli, G. and Sannipoli, R.},
  journal={Pacific Journal of Mathematics},
  volume={320},
  number={2},
  pages={241--259},
  year={2023},
  publisher={Mathematical Sciences Publishers}
}

@article{gavitone2023steklov,
  title={On a {S}teklov-{R}obin eigenvalue problem},
  author={Gavitone, N. and Sannipoli, R.},
  journal={Journal of Mathematical Analysis and Applications},
  volume={526},
  number={2},
  pages={127254},
  year={2023},
  publisher={Elsevier}
}

@article{basak2026bounds,
  title={Bounds for higher {S}teklov and mixed {S}teklov {N}eumann eigenvalues on domains with holes},
  author={Basak, S. and Verma, S.},
  journal={Annali di Matematica Pura ed Applicata (1923-)},
  pages={1--31},
  year={2026},
  publisher={Springer}
}

@article{anoop2020reverse,
  title={On reverse {F}aber-{K}rahn inequalities},
  author={Anoop, TV and Kumar, K A.},
  journal={Journal of Mathematical Analysis and Applications},
  volume={485},
  number={1},
  pages={123766},
  year={2020},
  publisher={Elsevier}
}

@article{anoop2021shape,
  title={A shape variation result via the geometry of eigenfunctions},
  author={Anoop, TV and Kumar, K A. and Kesavan, S.},
  journal={Journal of Differential Equations},
  volume={298},
  pages={430--462},
  year={2021},
  publisher={Elsevier}
}

@article{hong2020shape,
  title={Shape monotonicity of the first {S}teklov-{D}irichlet eigenvalue on eccentric annuli},
  author={Hong, J. and Lim, M. and Seo, D.},
  journal={arXiv preprint arXiv:2007.10147},
  year={2020}
}

@article{gavitone2024monotonicity,
  title={A monotonicity result for the first {S}teklov--{D}irichlet {L}aplacian eigenvalue},
  author={Gavitone, N. and Piscitelli, G.},
  journal={Revista Matem{\'a}tica Complutense},
  volume={37},
  number={2},
  pages={509--523},
  year={2024},
  publisher={Springer}
}

@article{GloriaPaoli2021C,
title = {A stability result for the {S}teklov {L}aplacian Eigenvalue Problem with a spherical obstacle},
journal = {Communications on Pure and Applied Analysis},
volume = {20},
number = {1},
pages = {145-158},
year = {2021},
issn = {1534-0392},
doi = {10.3934/cpaa.2020261},
url = {https://www.aimsciences.org/article/id/016bca28-5f0f-4cae-aa66-49a901b67ecb},
author = {Paoli, G. and Piscitelli, G. and Sannipoli, R.}
}

@article{sannipoli2025estimates,
  title={Estimates for the First and Second {S}teklov--{D}irichlet Eigenvalues: R. Sannipoli},
  author={Sannipoli, R.},
  journal={Milan Journal of Mathematics},
  volume={93},
  number={2},
  pages={435--453},
  year={2025},
  publisher={Springer}
}

@article{basak2026sharp,
  title={Sharp bounds and geometric properties of the first non trivial {S}teklov {N}eumann Eigenvalue},
  author={Basak, S. and Paoli, G. and Sannipoli, R. and Verma, S.},
  journal={arXiv preprint arXiv:2603.25448},
  year={2026}
}

@article{banuelos2010eigenvalue,
  title={Eigenvalue inequalities for mixed {S}teklov problems},
  author={Banuelos, R. and Kulczycki, T. and Polterovich, I. and Siudeja, B.},
  journal={Operator theory and its applications},
  volume={231},
  pages={19--34},
  year={2010},
}

@article{hassannezhad2020eigenvalue,

  title={Eigenvalue bounds of mixed {S}teklov problems},
  author={Hassannezhad, A. and Laptev, A.},
  journal={Communications in Contemporary Mathematics},
  volume={22},
  number={02},
  pages={1950008},
  year={2020},
  publisher={World Scientific}
}

@article{arias2024applications,
  title={Applications of possibly hidden symmetry to {S}teklov and mixed {S}teklov problems on surfaces},
  author={Arias-Marco, T. and Dryden, E. B and Gordon, C. S and Hassannezhad, A. and Ray, A. and Stanhope, E.},
  journal={Journal of Mathematical Analysis and Applications},
  volume={534},
  number={2},
  pages={128088},
  year={2024},
  publisher={Elsevier}
}

@article{CITO2026114086,
title = {Optimality and stability of the radial shapes for the Sobolev trace constant},
journal = {Nonlinear Analysis},
volume = {269},
pages = {114086},
year = {2026},
issn = {0362-546X},
doi = {https://doi.org/10.1016/j.na.2026.114086},
url = {https://www.sciencedirect.com/science/article/pii/S0362546X26000325},
author = { Cito,S.}
}

@article{cito2025stability,
  title={A stability result for the first {R}obin--{N}eumann eigenvalue: A double perturbation approach},
  author={Cito, S. and Paoli, G. and Piscitelli, G.},
  journal={Communications in Contemporary Mathematics},
  volume={27},
  number={06},
  pages={2450039},
  year={2025},
  publisher={World Scientific}
}

@article{paoli2020sharp,
  title={Sharp estimates for the first p-{L}aplacian eigenvalue and for the p-torsional rigidity on convex sets with holes},
  author={Paoli, G. and Piscitelli, G. and Trani, L.},
  journal={ESAIM: Control, Optimisation and Calculus of Variations},
  volume={26},
  pages={111},
  year={2020},
  publisher={EDP Sciences}
}
\end{document}